\documentclass[reqno]{amsart}

\usepackage{natbib}

\usepackage{graphicx} % Required for inserting images
\usepackage{amssymb}
\usepackage{amsmath}
\usepackage{proof}
\usepackage{enumitem}

\usepackage{amsthm}
\newtheorem{definition}{Definition}
\newtheorem{theorem}[definition]{Theorem}

\newtheorem{proposition}[definition]{Proposition}

\newcommand{\proves}{\vdash}
\newcommand{\supp}{\Vdash}

\newcommand{\base}[1]{\mathfrak{#1}}

\usepackage{color}

\renewcommand{\phi}{\varphi}
\renewcommand{\epsilon}{\varepsilon}

\newcommand{\pa}{\textsf{PA}}
\newcommand{\robq}{\textsf{Q}}
\newcommand{\eq}{\textsf{EQ}}

\usepackage{setspace}
\title{Classical Arithmetic without Bivalence}

\usepackage{amsaddr}
\author{Alexander V. Gheorghiu}
\address{University of Southampton, UK \and University College London, UK}
\email{a.v.gheorghiu@soton.ac.uk}

\begin{document}

\begin{abstract}
    Sandqvis's semantics for classical logic without bivalence resolves the question of an anti-realist account of classical reasoning after Dummett. This paper applies the framework to the essential questions of metamathematics. The system intuitively handles $\omega$-incompleteness, makes induction meaning-constitutive, and yields an elementary consistency proof for Peano Arithmetic using only ordinary induction on the natural numbers, with no appeal to transfinite ordinals or recognition-transcendent truth. 
\end{abstract}


\maketitle

\section{Introduction}

Traditionally, formal logic has been developed to serve metamathematics --- that is,   the study of mathematics itself using mathematical methods. This has yielded broadly two views: formalism and denotationalism.  The first regards mathematical theories as systems of rules for manipulating uninterpreted symbols; questions about consequence become questions about what can be derived in a proof system \citep{hilbertAckermann1968principles,gentzen1969investigations}. The second treats mathematical vocabulary as referring to mind-independent structures, with logical consequence identified with truth-preservation across all models \citep{Tarski1936}. Both perspectives are powerful: each supports familiar results concerning incompleteness, independence and consistency. Both are problematic as foundations.  

Formalism suffers from certain awkward limitations. A simple example is $\omega$-incompleteness --- a theory $\Gamma$ may prove $\phi(n)$ for any numeral $n$ but fail to prove $\forall x\phi(x)$. How can this be when the ontology is supposed to be determined by the theory? These things are accepted as brute features of a chosen calculus.    While denotationalism can explain these limitations (e.g., by saying that there may be unnamed entities under the scope of the quantifier),  it comes at the cost of presupposing a realist metaphysics and a truth-conditional theory of meaning. \citep{dummett1959truth,dummett1963realism} and  \cite{vaananen2012second} have argued that this begs the question of a foundation of logic is supposed to serve.

Recently, another alternative has been developed: inferentialism. This is the view that meaning is grounded in inferential role and warranted assertibility. It's roots are in in the works of \cite{dummett1993theoryI,dummett1991logicalbasis} and it has been developed most generally by \cite{brandom1994making}. He called justifying classical reasoning on this view `one of the most fundamental and intractable problems in the theory of meaning; indeed, in all philosophy'. This challenge was answered by Sandqvist in  ``Classical Logic without Bivalence'' (\emph{Analysis}, 2009). 

The present paper now asks what this inferentialist semantics for classical logic has to offer the debates on metamathematics. To this end, we establishes the consistency of \emph{Peano Arithmetic} ($\pa$) within the inferential semantics. The proof uses only induction over the natural numbers and no transfinite ordinals (as required in formalism) or mind-independent ontology (as required in denotationalism). On the inferentialist view, the behaviour of arithmetical vocabulary is defined by the theory of arithmetic. When that theory is sufficiently strong to know that the numerals are the ontology, problems like $\omega$-incompleteness vanish immediately.

\section{Classical Logic without Bivalence}

We begin by recalling the relevant definitions from  \cite{Sandqvist_2009}. Fix a first-order signature $\sigma$. The atomic formulas are predicates applied to terms; the logical constant $\bot$ is not considered atomic.

\begin{definition}[Atomic Rule]
An \emph{atomic rule} is an inference figure of the form
\[
\frac{P_1 \quad \cdots \quad P_n}{C}
\]
where $C, P_1, \ldots, P_n$ are atomic formulas.
\end{definition}

\begin{definition}[Base]
A \emph{base} $\base{B}$ is a set of atomic rules. We write $\supp_{\base{B}} A$ for derivability of an atom $A$ in $\base{B}$ as the composition of atomic rules.
\end{definition}

Intuitively, a base represents an agent's material inferential commitments --- the non-logical inferences they accept. The meaning of the logical constants is fixed by a support judgment. 

\begin{definition}[Support]
The \emph{support} judgment $\supp$ is defined inductively by the clauses in Figure~\ref{fig:support}, where $\base{B}$ and $\base{C}$ are bases, all formulas are closed, and $\Delta$ is a non-empty finite set of closed formulas.

\begin{figure}
\hrule
\begin{align*}
&\supp_{\base{B}} A &&\text{iff } \qquad \supp_{\base{B}} A \tag{At}\\
&\supp_{\base{B}} \phi \to \psi &&\text{iff } \qquad \phi \supp_{\base{B}} \psi \tag{$\to$}\\
&\supp_{\base{B}} \forall x\phi &&\text{iff } \qquad \supp_{\base{B}} \phi[x \mapsto t] \text{ for all closed terms } t \tag{$\forall$}\\
&\supp_{\base{B}} \bot &&\text{iff } \qquad \supp_{\base{B}} A \text{ for every closed atom } A \tag{$\bot$}\\
\Delta &\supp_{\base{B}} \phi &&\text{iff } \qquad \text{for all } C \supseteq B, \text{ if } \supp_{\base{C}} \psi \text{ for all } \psi \in \Delta, \text{ then } \supp_{\base{C}} \phi \tag{Inf}
\end{align*}
\vspace{0.5em}
\hrule
\caption{Support in a Base}
\label{fig:support}
\end{figure}

For an arbitrary (possibly infinite) set $\Gamma$ of formulas:
\[
\Gamma \supp \phi \qquad \text{ iff } \qquad \Delta \supp \phi \text{ for some finite } \Delta \subseteq \Gamma.
\]
\end{definition}

This completes the background on the semantics. We defer to \cite{Schroeder-Heister_2024} for details on the motivation and development of proof-theoretic semantics in general.

\section{Theories of Arithmetic}

The language of arithmetic has signature $\sigma=\langle 0,S,+,\cdot\rangle$ together with the equality predicate $=$. All theories considered satisfy the usual equality axioms:
\begin{align}
    & \forall x(x=x) && \tag{\eq1}\\
    & \forall x,y(x=y\to y=x) && \tag{\eq2}\\
    & \forall x,y,z\bigl(x=y\to (y=z\to x=z)\bigr) && \tag{\eq3}\\
    & \forall x,y\bigl(x=y\to(\phi(x)\to \phi(y))\bigr) && \tag{\eq4($\phi$)}
\end{align}
where $\eq4(\phi)$ is the substitution principle for each formula $\phi$.

Let us now return to the question of $\omega$-incompleteness. In the present semantics, the universal quantifier does not free to float above what is determined by theory, rather it is restricted to the syntax. This is because in inferentialist, what exists is determined by the inferential roles of terms in our language. Consequently, all that is is required for a theory $\Gamma$ to be $\omega$-complete is that it knows that the domain is captured by the numerals. 

\begin{definition}[Numerically Definite]
A theory of arithmetic $\Gamma$ is \emph{numerically definite} if for every closed term $t$ there exists a numeral $\bar n$ such that
\[
\Gamma \supp (t=\bar n).
\]
\end{definition} 
\begin{proposition}
If $\Gamma$ is numerically definite, then it is $\omega$-complete. That is, if $\Gamma \supp \phi(\bar n)$ for every numeral $\bar n$, then $\Gamma \supp \forall x\,\phi(x)$.
\end{proposition}

\begin{proof}
Assume $\Gamma \supp \phi(\bar n)$ for each numeral $\bar n$. Let $t$ be an arbitrary closed term. By numerical definiteness there exists a numeral $\bar n$ such that $\Gamma \supp (t=\bar n)$, and hence (in particular) $\Gamma \supp (t=\bar n)$. Using $\eq4(\phi)$ we may substitute equals for equals: from $\Gamma \supp (t=\bar n)$ and $\Gamma \supp \phi(\bar n)$ we obtain $\Gamma \supp \phi(t)$. Since $t$ was an arbitrary closed term, the universal-clause yields $\Gamma \supp \forall x\,\phi(x)$.
\end{proof}

Observe that induction is a consequence of such theories. Induction is the statement that surveying the numerals suffices for the whole space. This may be viewed in the light of Tennant's~(\citeyear{tennant1984disproofs}) interpretation against Weir's~(\citeyear{weir1983truth}) behaviourist realism: where Weir seeks extensional truth-functions for quantifiers (which fail on universals), Tennant holds that grasp of arithmetical meaning manifests through proof-related practice governed by induction.

\begin{proposition}
Let $\Gamma$ be numerically definite. If
\[
\Gamma \supp \phi(0)
\qquad\text{and}\qquad
\Gamma \supp \forall x\bigl(\phi(x)\to \phi(S(x))\bigr),
\]
then
\[
\Gamma \supp \forall x\,\phi(x).
\]
\end{proposition}

\begin{proof}
From $\Gamma \supp \forall x(\phi(x)\to \phi(S(x)))$ we obtain, in particular, $\Gamma \supp \phi(\bar n)\to \phi(S(\bar n))$ for each numeral $\bar n$ (by the universal-clause). Together with $\Gamma \supp \phi(0)$, repeated applications of modus ponens yield
\[
\Gamma \supp \phi(\bar n)
\]
for every numeral $\bar n$.

Now let $t$ be any closed term. By numerical definiteness, choose a numeral $\bar n$ such that $\Gamma \supp (t=\bar n)$. Since $\Gamma \supp \phi(\bar n)$, substitution of equals for equals via $\eq4(\phi)$ yields $\Gamma \supp \phi(t)$. As $t$ was arbitrary, the universal-clause for $\forall$ gives $\Gamma \supp \forall x\,\phi(x)$.
\end{proof}

One objection to the foregoing is that we appear to have \emph{collapsed} the familiar distinction between Robinson Arithmetic ($\mathsf{Q}$) and Peano Arithmetic ($\mathsf{PA}$). Is this a defect?

The $\mathsf{Q}$–$\mathsf{PA}$ gap is proof-theoretic rather than semantic. It presupposes a model-theoretic picture on which the universal quantifier ranges over items beyond the numerals themselves. In the present semantics, however, the theory itself \emph{fixes} the ontology, and the quantifier therefore ranges only over the terms of the language. There is no verification-transcendent residue beyond what the theory's inferential commitments make manifest, and so no room in which the $\mathsf{Q}$–$\mathsf{PA}$ gap could arise. This is a feature, not a defect. 

The inferentialist setup offers another curious phenomenon. A theory of arithmetic $\Gamma$ may admit induction without being numerically definite. This would mean that $\Gamma$ knows its ontology is determined by the numerals yet lacks the calculating power to determine which closed terms denote which numerals. Such a theory should be regarded as incomplete. 

\section{The Consistency of Arithmetic}

The preceding sections have shown that Sandqvist's~(\citeyear{Sandqvist_2009}) semantics exhibits some distinctive features when viewed as a metamathematical framework. We now show that it can directly answer the core issue of proving a theory to be consistent.  

\begin{definition}[Consistency]
A theory $\Gamma$ is \emph{consistent} iff $\Gamma \not \supp \bot$.
\end{definition}

We shall work with the standard axiomatization of arithmetic, \emph{Peano Arithmetic} ($\pa$). Although $\pa$ is indistinguishable from $\robq$ in the semantics, the two theories are syntactically distinct.

\begin{definition}[Peano Arithmetic]
	The axioms of Peano Arithmetic ($\pa$) extend the equality axioms with the following schemata, for all formulae $\phi$:
\begin{align}
& \forall x(\neg (S(x) = 0)) \tag{PA1} \\
& \forall x,y \big( S(x) = S(y) \rightarrow x = y \big) \tag{PA2} \\
& \forall x \big(x + 0 = x\big) \tag{PA3} \\
& \forall x,y \big(x + S(y) = S(x + y)\big) \tag{PA4} \\
& \forall x \big(x \cdot 0 = 0\big) \tag{PA5} \\
& \forall x,y \big(x \cdot S(y) = x \cdot y + x\big) \tag{PA6} \\
&
\big( \phi(0) \land \forall x\, (\phi(x) \rightarrow \phi(S(x))) \big) \rightarrow \forall x\, \phi(x)
\tag{PA7($\phi$)}
\end{align}
\end{definition}

\begin{theorem}
$\pa$ is consistent.
\end{theorem}

\begin{proof}
To show $\pa \not \supp \bot$, it suffices to exhibit a base $\base{X}$ such that $\not \supp_{\base{X}} \bot$ while $\supp_{\base{X}} \alpha$ for every axiom $\alpha \in \pa$.

Let $\base{A}$ be the arithmetic base defined in Figure~\ref{fig:arithmetic-base}, where $A$ ranges over all atomic formulas and $x,y$ range over closed terms. With the possible exception of the substitution schema ($\eq4$), it is easy to see that $\base{A}$ supports the equality axioms. Similarly, with the possible exception of the induction schema ($\pa7$), it is immediate that $\base{A}$ supports the axioms of $\pa$. That it also supports every instance of the substitution and induction schemas follows by straightforward induction on the structure of terms. It therefore remains only to show that $\base{A}$ itself is consistent; that is, $\not \supp_{\base{A}} \bot$.

To this end, define a \emph{weight function} $w(t)$ on closed terms $t$ of the language of arithmetic by primitive recursion:
\[ 
w(t) :=
\begin{cases}
0 & \text{if } t = 0, \\
w(t') + 1 & \text{if } t = S(t'), \\
w(t_1) + w(t_2) & \text{if } t = t_1 + t_2, \\
w(t_1)\, w(t_2) & \text{if } t = t_1 \cdot t_2.
\end{cases}
\]
A straightforward induction on the structure of $\base{A}$-derivations shows that if an equation $t_1 = t_2$ occurs in a $\base{A}$-derivation, then $w(t_1) = w(t_2)$.

It follows immediately that $\base{A}$ cannot support any false numerical identity. In particular, if $\supp_{\base{A}} \bot$, then $\proves_{\base{A}} (1 = 0)$. But this is impossible, since it would imply $w(1) = w(0)$ --- that is, $1 = 0$.

Hence, $\not \supp_{\base{A}} \bot$, and therefore $\pa$ is consistent.
\end{proof}

\begin{figure}  
\hrule \vspace{2mm} 
\[ 
\begin{array}{ccc} \infer[\pa1]{A}{S(x) = 0} & \infer[\pa2]{x=y}{S(x) = S(y)} & \infer[\pa3]{x+0=x}{} \\[4mm] \infer[\pa4]{x+S(y)=S(x+y)}{} & \infer[\pa5]{x\cdot 0=0}{} & \infer[\pa6]{x\cdot S(y)=x\cdot y+x}{} \end{array} 
\] 
\dotfill 
\[ 
\infer[\eq1]{x=x}{} \qquad \infer[\eq2]{x=y}{y=x} \qquad \infer[\eq3]{x=z}{x=y & y=z}
\] 
\hrule \caption{Arithmetic Base $\base{A}$} \label{fig:arithmetic-base}  
\end{figure}

It is also important to stress that the construction of the base $\base{A}$ is not a mere stipulation of consistency. The base is defined explicitly, using entirely standard and acceptable mathematical techniques, and there is in general no guarantee that such a construction will be consistent. Indeed, arbitrary bases of this kind may perfectly well be inconsistent. A substantial part of the foregoing argument is devoted precisely to showing that the particular base $\base{A}$ constructed here both supports the axioms of $\pa$ and yet does not support $\bot$.

It is also worth emphasising that the preceding proof uses only induction on the natural numbers and no transfinite principles. This is entirely acceptable, since the background framework is inferentialist rather than formalist: the definition of the support relation and the reasoning about it may legitimately make use of general, and in particular inductive, reasoning. This is what Dummett~\cite{Dummett_1978} calls a pragmatic (as opposed to vicious) circularity. 

Importantly, the result should not be viewed as a contribution to Hilbert's Programme, nor as showing that the consistency of $\pa$ can be established in a metatheory that is, by formalist standards, weaker or more secure than $\pa$ itself. We shall now consider two objections to the above consistency proof.

The first, from a formalist, is that the definition of the support judgement involves an alternation of quantifiers and negations whose complexity far exceeds that of the arithmetical theories under discussion. One might conclude that the resulting semantics is no better, as a metatheory, than the familiar model-theoretic framework it is meant to replace. 

This objection has force only if one tacitly treats formalism as the default standard for metatheory. Yet any metatheoretical framework stands in need of philosophical justification; finitist and formalist positions are coherent and important, but failt to accommodate the full range of ordinary mathematical reasoning (see e.g., G\"odel's incompleteness theorems \citep{godel1931undecidable,tait_finitism_1981}). To adopt the present semantics is not to smuggle it into a formalist background, but to adopt the inferentialist conception of meaning articulated by \cite{dummett1991logicalbasis} and \cite{Sandqvist2005inferentialist}. On this conception, the support relation is not a single formula to be arithmetised but a meta-theoretic explanation of what we mean when we use our logical language. Questions of arithmetical definability are therefore beside the point; an objection here is ultimately an objection to inferentialism as a theory of meaning.

A second objection comes from the model-theoretic side. The consistency proof appeals to the natural numbers and their properties, in particular through ordinary mathematical induction. Does this not merely displace the problem? What, on the present approach, licenses this appeal? This worry rests on a misunderstanding of the inferentialist standpoint. On the model-theoretic picture, mathematical statements are made true or false by mind-independent structures, and induction is justified by the fact that the natural numbers, so conceived, possess a certain objective structure. On the present approach, by contrast, to grasp the concept of a natural number just is to grasp the inferential role of arithmetical vocabulary. Induction is not a substantive hypothesis \emph{about} an antecedently given domain; it is constitutive of our understanding of the natural numbers.

\section{Soundness and Completeness}
We have shown that $\pa$ is consistent in the inferentialist semantics. Does this suffice to show that $\pa$ is consistent in classical logic? Let $\vdash$ denote classical derivability. It is sound with respect to the semantics: if $\Gamma \vdash \varphi$ then $\Gamma \Vdash \varphi$ \citep{Sandqvist_2009,Gheorghiu2025-GHEPSF-4}. Hence, by contraposition, from
\[
\pa \not\Vdash \bot
\]
we may infer
\[
\pa \not\vdash \bot.
\]
In this sense, the consistency result transfers immediately to the classical setting. 

One may nevertheless ask what happens if one works with a version of the base semantics for which classical derivability is not only sound but also complete:
\[
\Gamma \Vdash \varphi \quad \text{iff} \quad \Gamma \vdash \varphi.
\]
For this to be possible, the language must contain an infinite reserve of constants disjoint from the theory. In the completeness proof of \cite{Gheorghiu2025-GHEPSF-4} the unused constants serve as eigenvariables in the completeness proof. 

Accordingly, consider an extension of the language of arithmetic with an indefinite sequence of constants, so that the signature becomes
\[
\sigma \;=\; \langle (c_i)_i, S, +, \cdot \rangle,
\]
with $c_0$ distinguished as the zero numeral $0$. In this setting, soundness and completeness obtain, and $\Vdash$ and $\vdash$ are extensionally equivalent. What becomes of the foregoing consistency argument under this modification? Essentially nothing changes. 

We extend the base $\mathcal{A}$ by adding, for each constant $c_i$, the rule
\[
\infer{0 = c_i}{}
\]
and we extend the weight function by stipulating $w(c_i) = 0$ for every $i \in \mathbb{N}$. One may view this as saying that $\pa$, which is no longer numerically definite in this enriched signature, is still supported by a numerically definite base; intuitively, the additional constants are treated as mere alternative names for zero. With these minor adjustments, the consistency argument goes through exactly as before and remains entirely elementary. 

A final subtlety concerns the surrounding ontology. The definitions used in formulating the semantics are, in keeping with contemporary mathematical practice, couched in set-theoretic terms. But the inferentialist standpoint underwriting the framework is a form of anti-realism. The `infinite' collections that occur in the syntax are therefore not to be read as completed totalities, but as \emph{potential} infinities: to say that there is an infinite sequence $(c_i)$ of constants is not to posit a completed infinite stock of symbols, but only that there is an indefinite supply of them, in the sense that at any stage a new symbol can be introduced. This matches ordinary mathematical practice, where new signs may always be added as needed. 

\bibliographystyle{apalike}
\bibliography{bib}

\end{document}